\documentclass{amsart}

\usepackage{amsfonts,amssymb,amscd,amsmath,enumerate,verbatim}

\newcommand\Ass{\operatorname{Ass}}
\newcommand\bss{{\boldsymbol s}}
\newcommand\bsx{{\boldsymbol x}}
\newcommand\BZ{{\mathbb Z}}
\newcommand{\cd}[2]{\operatorname{codepth}_{#1}{#2}}

\newcommand{\col}{\colon}

\newcommand\depth{\operatorname{depth}}
\newcommand\Ext{\operatorname{Ext}}
\newcommand\fd{\operatorname{fd}}
\newcommand\fm{{\mathfrak m}}
\newcommand\fn{{\mathfrak n}}
\newcommand\fp{{\mathfrak p}}
\newcommand\fq{{\mathfrak q}}

\newcommand\Hom{\operatorname{Hom}}

\newcommand{\HH}[2]{\operatorname{H}_{#1}(#2)}

\newcommand\Ker{\operatorname{Ker}}

\newcommand\length{\operatorname{length}}
\newcommand{\lra}{\longrightarrow}

\newcommand\Spec{\operatorname{Spec}}
\newcommand\Supp{\operatorname{Supp}}
\newcommand{\tensor}[3][R]{{#3}^{\hskip-.5pt\otimes_{\hskip-1.4pt #1}^{\hskip-1pt #2}}}
\newcommand{\tf}[2]{{\boldsymbol\bot}_{#1}{#2}}
\newcommand\tor[4]{\operatorname{Tor}^{#2}_{#1}(#3,#4)}
\newcommand\Tor{\operatorname{Tor}}

\newcommand{\xla}{\xleftarrow}
\newcommand{\xra}{\xrightarrow}
\newcommand{\zdr}[2]{{\boldsymbol\top}_{\hskip-2pt #1}{#2}}

\newtheorem{theorem}{Theorem}[section]
\newtheorem*{Theorem}{Main Theorem}
\newtheorem{proposition}[theorem]{Proposition}
\newtheorem{lemma}[theorem]{Lemma}

\theoremstyle{definition}

\newtheorem{example}[theorem]{Example}
\newtheorem{remark}[theorem]{Remark}

\theoremstyle{remark}
\newtheorem*{Notes}{Notes}
\newtheorem*{Remark}{Remark}

\newtheorem{chunk}[theorem]{}
\newenvironment{bfchunk}{\begin{chunk}\textbf}{\end{chunk}}

\begin{document}

\title[Detecting flatness]{Detecting flatness over smooth bases}

\author[L.~L.~Avramov]{Luchezar L.~Avramov}
\address{Department of Mathematics,
University of Nebraska, Lincoln, NE 68588, U.S.A.}
\email{avramov@math.unl.edu}

\author[S.~B.~Iyengar]{Srikanth B.~Iyengar}
\address{Department of Mathematics,
University of Nebraska, Lincoln, NE 68588, U.S.A.}
\email{iyengar@math.unl.edu}
\thanks{Research partly supported by NSF grants 
DMS-0803082 (LLA) and DMS-0903493 (SBI)}

\date{20th January 2011}

\begin{abstract}
Given an essentially finite type morphism of schemes $f\colon X\to
Y$ and a positive integer $d$, let $f^{\{d\}}\colon X^{\{d\}}\to Y$
denote the natural map from the $d$-fold fiber product $X^{\{d\}}=
X\times_{Y}\cdots\times_{Y}X$ and $\pi_i\colon X^{\{d\}}\to X$ the $i$th
canonical projection.  When $Y$ smooth over a field and $\mathcal F$
is a coherent sheaf on $X$, it is proved that $\mathcal F$ is flat
over $Y$ if (and only if) $f^{\{d\}}$ maps the associated points of
${\bigotimes_{i=1}^d}\pi_i^*{\mathcal F}$ to generic points of $Y$, for
some $d\ge\dim Y$.  The equivalent statement in commutative algebra is an
analog---but not a consequence---of a classical criterion of Auslander
and Lichtenbaum for the freeness of finitely generated modules over
regular local rings.
\end{abstract}

\maketitle

\section*{Introduction}

Knowledge that a scheme can be fibered as a flat family over some regular
base scheme represents fundamental structural information.  For this
reason---among others---it is desirable to have efficient methods for
deciding the flatness of a morphism of \emph{noetherian} schemes $f\col
X\to Y$ with $Y$ \emph{regular}.  

One necessary condition is that $f$ maps associated points of $X$ to 
generic points of $Y$.  When $\dim Y = 1$, this condition is also 
sufficient; see~\cite[III.9.7]{Ha}.

When the morphism $f$ is \emph{finite}, a criterion for freeness of finite modules
over regular local rings, due to Auslander \cite{Au} and Lichtenbaum
\cite{Li}, translates into a similar criterion involving $d$-fold fiber products
$X^{\{d\}}=X\times_{Y}\cdots\times_{Y}X$: If $d\geq \dim Y$ and 
the natural morphism $f^{\{d\}}\col X^{\{d\}}\to Y$ sends associated 
points of $X^{\{d\}}$ to generic points of $Y$, then $f$ is flat.
This was observed by Vasconcelos~\cite{Va}, who proved that the 
conclusion holds for all morphisms $f$ when $d=2=\dim Y$, and 
conjectured that it holds in all dimensions when $f$ is a morphism 
essentially of finite type.  

We prove a criterion for flatness relative to $f$ of coherent sheaves
$\mathcal F$ on $X$, using the canonical projections $\pi_i\col
X^{\{d\}}\to X$:  \emph{Assume that $Y$ is essentially smooth over
some field and $f$ is essentially of finite type.  If $f^{\{d\}}$ maps
the associated points of ${\bigotimes_{i=1}^d}\pi_i^*{\mathcal F}$ to
generic points of\,\ $Y$ for some $d\ge\dim Y$, then $\mathcal{F}$ is
flat over~$Y$.}  Thus, flatness is detected by the values of $f^{\{d\}}$
at finitely many points.  Setting $\mathcal{F}=\mathcal{O}_Y$ one obtains
a proof of Vasconcelos' conjecture in a case of prime interest in geometry.

The theorem above is an analog of Auslander's criterion.  It is equivalent
to the following statement in commutative algebra, proved in Section \ref{sec:flatness}.

\begin{Theorem}
Let $K$ be a field, $R$ an essentially smooth $K$-algebra, $A$ an algebra 
essentially of finite type over $R$, and $M$ a finite $A$-module.

If $\tensor dM$ is torsion-free over $R$ for some $d\geq \dim R$, then $M$ is flat over $R$.
 \end{Theorem}

For $K=\mathbb C$, the preceding result is proved by Adamus, Bierstone,
and Milman in \cite{ABM}, subsequent to the special case $M=A$ with $A$
equidimensional and of finite type over $\mathbb C$, settled by Galligo
and Kwieci\'nski~\cite{GK}.  In those papers the algebraic statements
are deduced from analogous theorems about morphisms of complex-analytic
spaces, proved by using transcendental-geometric tools.

Our proof is algebraic, and its architecture reflects Auslander's design.  The 
argument proceeds in four steps, summed up in the opening statements 
of the first four sections.  Sections \ref{sec:rigidity} and \ref{sec:flatness}  
are mostly homological in nature.  They deal with vanishing of $\Tor$ modules---and 
thus, ultimately, with flatness---and allow for an adaptation of methods used by 
Auslander in~\cite{Au}.  

In Sections \ref{sec:torsion} and \ref{sec:torsion2} we relate vanishing of 
$\operatorname{Tor}$ modules to torsion-freeness.  These sections form
the core of the paper.  The techniques applied in them are, by necessity,
different from those that work for finite modules over local rings.
More detailed comparisons are given in notes immediately following 
the first theorem in each section.  In these notes we also explain why in the Main 
Theorem the base ring $R$ is assumed to be an essentially smooth algebra 
over a field, rather than just a regular ring.

In  Section~\ref{sec:two} we focus on rings of dimension two. Assuming only that 
$R$ is regular, we establish a criterion for flatness that covers a larger class of 
modules than those in the Main Theorem, and includes Vasconcelos' result.

\section{Rigidity}
  \label{sec:rigidity}

In this paper rings are assumed commutative.

An algebra $B$ over a ring $K$ is said to be \emph{essentially of finite
type} if it is a localization of some finitely generated $K$-algebra.
In case $K$ is a field, $B$ is \emph{essentially smooth over $K$} if,
in addition, the ring $B\otimes_KK'$ is regular for every field extension
$K\subseteq K'$; thus, when the field $K$ is perfect, the $K$-algebra 
$B$ is essentially smooth if and only if $B$ is a regular ring.

For the rest of this section $R$ denotes a noetherian ring.  

For every prime ideal $\fp$ of $R$ we set $k(\fp)=R_\fp/\fp R_\fp$.

We say that an $R$-module $M$ is \emph{essentially of finite type} if there 
exists an $R$-algebra $A$ essentially of finite type with the property that 
$M$ is a finite $A$-module, and the $R$-module structure induced through
$A$ coincides with the original one. Any such algebra $A$ will be called a 
\emph{witness} for $M$.

It is clear that finite $R$-modules are essentially of finite type, and that the latter
class is much larger than the former. Remark~\ref{ex:noneft} describes an 
interesting family of modules that are \emph{not} essentially of finite type.

\begin{theorem}
\label{thm:rigidity}
Let $R$ be an essentially smooth algebra over a field, and let $M$ and 
$N$ be $R$-modules essentially of finite type.

If $\tor iRMN=0$ for some $i\geq 0$, then $\tor jRMN=0$ for each $j\geq i$. 
\end{theorem}

  \begin{Notes}
When $R$ is a regular local ring and $M$ and $N$ are finite $R$-modules,
the conclusion above was proved by Auslander \cite[2.1]{Au} in case $R$ 
is an algebra over some field, by using Koszul complexes, and 
extended to the general case by Lichtenbaum \cite[Cor.\,1]{Li}, by 
applying different techniques.

Our proof of Theorem \ref{thm:rigidity} also relies on Koszul complexes.
  \end{Notes}

Given a finite sequence $\bsx$ of elements of $R$ and an $R$-module 
$M$, set
\[
\HH i{\bsx; M}=\HH i{R\langle\bsx\rangle\otimes_{R}M}\,,
\]
where $R\langle\bsx\rangle$ is the Koszul complex on $\bsx$.  We recall
an important fact:

\begin{bfchunk}{Koszul rigidity.~I.}
\label{koszul:rigidity}
If $M$ is an $R$-module essentially of finite type, and $\HH i{\bsx; M}=0$ for some integer $i\ge0$, then $\HH j{\bsx; M}=0$ holds for $j\geq i$.

Indeed, let $A$ be a witness for $M$ and $\alpha\col R\to A$ the structure map. There is an isomorphism of complexes
$R\langle\bsx\rangle\otimes_{R}M\cong A\langle\alpha(\bsx)\rangle\otimes_{A}M$, and since $M$ is finite over
$A$, the result follows from the classical case; see \cite[2.6]{AB}.
  \end{bfchunk}

\begin{lemma}
\label{lem:smooth}
Let $K$ be a field and $R$ an essentially smooth $K$-algebra. 

Set $Q=R\otimes_KR$, let $\mu\col Q\to R$ be the surjective homomorphism 
of rings, given by $\mu(r\otimes r')=rr'$, and set $I=\Ker(\mu)$.

For every prime $\fq\in\Spec Q$ with $\fq\supseteq I$, any minimal generating 
set $\bsx$ for $I_\fq$, and each $j\in\BZ$ there is an isomorphism of $R_\fq$-modules
  \begin{equation*}
  \tor jRMN_{\fq} 
  \cong \HH j{\bsx; (M\otimes_KN)_{\fq}}\,.
  \end{equation*}
   \end{lemma}

  \begin{proof}
The first isomorphism is a localization of a formula from \cite[IX.4.4]{CE}:
  \[
\tor jRMN_{\fq} \cong \tor j{{Q}}R{M\otimes_{K}N}_{\fq}
\cong \tor j{Q_{\fq}}{R_\fq}{(M\otimes_KN)_{\fq}} \,.
  \]

For each prime ideal $\fp$ of $R$, set $k(\fp)=R_\fp/\fp R_\fp$.  The map 
$r\mapsto r\otimes1$ gives a flat homomorphism of rings 
$\iota\col R\to Q$ with fibers $k(\fp)\otimes_KR$.  Since 
$R$ is essentially smooth over $K$, both rings $R$ and $k(\fp)\otimes_KR$ are 
regular.  By \cite[23.7]{Ma}, the ring $Q$ is regular as well.
The induced map $Q_\fq\to R_\fq$ is a surjective homomorphism 
of regular local rings, so its kernel $I_\fq$ is generated by a regular 
sequence; see~\cite[14.2]{Ma}.  By~\cite[16.5]{Ma}, the Koszul complex
$Q_\fq\langle\bsx\rangle$ is a free resolution of $R_\fq$ over 
$Q_\fq$; this gives an isomorphism
\[
\tor j{Q_{\fq}}{R_\fq}{(M\otimes_KN)_{\fq}}\cong \HH j{\bsx; (M\otimes_KN)_{\fq}}\,.
\]
Concatenating the displayed isomorphisms completes the proof.
  \end{proof}

  \begin{proof}[Proof of Theorem~\emph{\ref{thm:rigidity}}]
By localization and Lemma \ref{lem:smooth}, it suffices to show
that $\HH i{\bsx; (M\otimes_KN)_{\fq}}=0$ implies $\HH j{\bsx;
(M\otimes_KN)_{\fq}}=0$ for $j\ge i$ and for each $\fq$ in $\Spec Q$ with
$\fq\supseteq I$.  This follows from \ref{koszul:rigidity}, for the module
$(M\otimes_{K}N)_{\fq}$ is essentially of finite type over $Q_{\fq}$, as 
witnessed by $(A\otimes_{K}B)_{\fq}$, where $A$ is a witness for
$M$ and $B$ is one for $N$.
  \end{proof}

  \begin{Remark}
    \label{not:eft}
The papers \cite{ABM,GK} deal with \emph{almost 
finite} modules over analytic algebras.  This class is distinct from 
that of modules essentially of finite type.
  \end{Remark}

\section{Torsion in tensor products}
  \label{sec:torsion}

Let $R$ be a ring and $U$ its multiplicatively closed subset consisting 
of all the non-zero-divisors.  The \emph{torsion submodule}, $\zdr RM$, 
of an $R$-module $M$ is the kernel of the localization map $M\to U^{-1}M$.  
There is an exact sequence 
\begin{equation}
\label{eq:zdr}
   \tag{2.0}
0\lra \zdr RM \lra M\lra \tf RM\lra 0
\end{equation}
of $R$-modules.  The module $M$ is said to be \emph{torsion} when 
$\zdr RM=M$; it is called \emph{torsion-free} when $\zdr RM=0$.  Note 
that $\zdr RM$ is a torsion module, while $\tf RM$ is a torsion-free one.

\begin{theorem}
\label{thm:torsion}
Let $R$ be an essentially smooth algebra over a field, and let 
$M$ and $N$ be $R$-modules essentially of finite type.

If the $R$-module $M\otimes_{R}N$ is torsion-free, then the following 
statements hold:
\begin{enumerate}[\quad\rm(1)]
\item $\tor iRMN=0$ for each $i\geq 1$.
\item $\tor iRM{\zdr RN}= 0 = \tor iR{\zdr RM}N$ for all $i\geq 0$.
\end{enumerate}
\end{theorem}

  \begin{Notes}
When $R$ is a regular local ring, $M$ and $N$ are finite $R$-modules, and 
$M\otimes_RN$ is torsion-free, the statements above hold by \cite[3.1(b)]{Au} 
(if $R$ is  unramified) and \cite[Cor.\,2(b)]{Li} (in general).  
The finiteness hypothesis is critical for the proofs of these results.

Our proof draws on different ideas.  The hypothesis that $R$ is an essentially 
smooth algebra is used at a crucial juncture of the argument, in order to 
replace certain $\Tor$ modules with appropriate Koszul homology modules.

The hypotheses of the theorem do not, in general, imply that $M$ or $N$ is
a torsion-free $R$-module; see Example \ref{ex:tensor1}.
  \end{Notes}

We start preparations for proving Theorem~\ref{thm:torsion} with a 
standard calculation.

\begin{lemma}
\label{lem:zdr}
Let $R$ be a ring and let $M$ and $N$ be $R$-modules. 

If $M\otimes_{R}N$ is torsion-free, then there are natural isomorphisms
\[
(\tf R{M})\otimes_{R} N \xla{\ \cong\ }  M\otimes_{R}N\xra{\ \cong\ } 
M\otimes_{R}(\tf R{N}) \xra{\ \cong\ } (\tf R{M})\otimes_{R}(\tf R{N})\,.
\]

If, in addition, $\tor 1R{\tf RM}N=0$, then $(\zdr R{M})\otimes_{R}N=0$.
\end{lemma}

\begin{proof}
Tensor the sequence \eqref{eq:zdr} with $N$ to get an exact sequence
\[
\tor 1R{\tf RM}N\lra (\zdr R{M})\otimes_{R}N\lra M\otimes_{R}N\xra{\ \tau\ }
(\tf R{M})\otimes_{R}N\lra 0
\]
As $(\zdr R{M})\otimes_{R}N$ is torsion and $M\otimes_{R}N$ is torsion-free, 
this sequence shows that $\tau$ is bijective.  By symmetry, so is 
$M\otimes_{R}N\to M\otimes_{R}(\tf R{N})$. Thus, $M\otimes_{R}(\tf R{N})$ 
is torsion-free, so the preceding argument shows that the homomorphism of $R$-modules 
$M\otimes_{R}(\tf R{N})\to(\tf R{M})\otimes_{R}(\tf R{N})$ is bijective.

The final assertion of the lemma is clear from the exact sequence above.
\end{proof}

\begin{lemma}
\label{lem:eft-tors}
If $M$ and $N$ are $R$-modules essentially of finite type, with witnesses 
$A$ and $B$, respectively, then for each $i\in\BZ$ the $R$-module 
$\tor iRMN$ is essentially of finite type, with witness $A\otimes_{R}B$.
\end{lemma}

\begin{proof}
By hypothesis, one has $A\cong A'/I$, where $A'$ is a localization of a 
polynomial ring over $R$. As $M$ is a finite $A'$-module, it has a resolution 
$F$ by free $A'$-modules  of finite rank.  This also is a resolution 
by flat $R$-modules, so for each $i\in\BZ$ one has 
$\tor iRMN\cong\HH i{F\otimes_RB}$.  Since $F\otimes_RB$ is a complex 
of finite $(A'\otimes_{R}B)$-modules, and the ring $A'\otimes_{R}B$ is 
noetherian, $\HH i{F\otimes_RB}$ is a finite $(A'\otimes_{R}B)$-module.
It is  annihilated by $I\otimes_RB$, so it has a structure of 
$(A\otimes_{R}B)$-module, which is necessarily finite.
  \end{proof}

We use elementary facts concerning depth.  These are not well documented for not necessarily finite modules, so we collect the statements we need. As usual, local rings are assumed to be noetherian.

\begin{bfchunk}{Depth.}
\label{ch:depth}
Let $S$ be a local ring, $\fn$ its maximal ideal, and $L$ an $S$-module.
Following Auslander and Buchsbaum \cite{AB}, we define the \emph{depth}
of $L$ by
\[
\depth_{S}L= e - \sup\{i\mid \HH i{\bss\,; L}\ne 0\}\,,
\]
where $\bss=s_{1},\dots,s_{e}$ is a generating set for $\fn$.  In particular, 
if $\HH i{\bss\,; L}=0$ for all $i$, then $\depth L = \infty$.
The definition of the Koszul complex $S\langle\bss\rangle$
gives
  \[
\HH e{\bss\,; L} = (0:\fn)_L \cong \Hom_{S}(S/\fn,L)\,.
 \]
The depth of $L$ equals the infimum of those integers $i$ with $\Ext^i_S(S/\fn,L)\ne0$; see \cite[\S1, Th\'eor\`eme 1]{Bo}.  As usual, we set 
$\depth S=\depth_SS$.  
   \end{bfchunk}

In applications, we need to track depth along ring homomorphisms.

  \begin{lemma}
  \label{lem:depth-change}
Let $S$ be a local ring, $\fn$ its maximal ideal, $\sigma\col S\to S'$ a homomorphism of rings, and $L'$ an $S'$-module.
  \begin{enumerate}[\rm(1)]
 \item
When $S'$ is local with maximal ideal $\fn'$, and $\sigma(\fn)\subseteq\fn'$,
then $\depth_{S'}L'=0$ implies $\depth_{S}L'=0$; the converse holds 
when $\fn S'=\fn'$.
\item 
If $\tau \col S'\to S''$ is a flat homomorphism, $\depth_{S}L'\le\depth_{S}(L'\otimes_{S'}S'')$,
and equality holds in case $\tau$ is also faithful.
  \end{enumerate}
 \end{lemma}

  \begin{proof}
(1) Set $k=S/\fn$ and $k'=S'/\fn'$.  In the string of $S$-linear maps
\[
\Hom_{S}(k,L')
\cong\Hom_{S'}(S'\otimes_{S}k,L')
\cong \Hom_{S'}(S'/\fn S',L')
\supseteq\Hom_{S'}(k',L')
\]
the isomorphisms are standard and the inclusion is induced by the surjection $S'/\fn S'\to S'/{\fn'}$;
it is an equality when $\fn S'=\fn'$. Now refer to \ref{ch:depth}.

(2) This follows from isomorphisms $\HH i{\bss\,; L'\otimes_{S'}S''}\cong\HH i{\bss\,; L'}\otimes_{S'}S''$.
  \end{proof}

When $M$ is an $R$-module $\Ass_RM$ denotes the set of prime ideals 
$\fp$ of $R$ for which there is a monomorphism $R/\fp\to M$, and $\Ass R$ stands
for $\Ass_RR$.  When $R$ is noetherian
$\Ass R$ is finite and contains every minimal prime of $R$.

Depth detects torsion through the following well-known observation:

\begin{lemma}
  \label{lem:localization}
Let $R$ be a noetherian ring and $M$ an $R$-module.  

The following condition implies that $M$ is torsion-free:
  \begin{equation}
 \label{eq:star}
   \tag{2.6.1}
\depth_{R_\fp}M_\fp>0 \quad\text{for every}\quad
\fp\in\Spec R\smallsetminus\Ass R
  \end{equation}
The converse assertion holds if every associated prime ideal of $R$ is minimal.
  \end{lemma}

  \begin{proof}
Recall: the set $U$ of non-zero-divisors of $R$ is equal to
$R\smallsetminus\bigcup_{\fq\in\Ass R}\fq$.

Assume that \eqref{eq:star} holds, but $um=0$ with $u\in U$ and 
$m\in M\smallsetminus\{0\}$.  Choose $\fp\in\Ass_R(Rm)$ and a
monomorphism $R/\fp\to Rm$.  Composed with the inclusion 
$Rm \subseteq M$, it induces an injection $R_\fp/\fp R_\fp\to M_\fp$, 
which gives $\depth_{R_\fp}M_\fp=0$.  Now \eqref{eq:star} implies 
$\fp\in\Ass R$, which is impossible, since $\fp$ contains the 
non-zero-divisor $u$.

When primes ideals in $\Ass R$ are minimal, prime avoidance shows that
each $\fp\notin\Ass R$ satisfies $\fp\cap U\neq\varnothing$.  If $M$ is torsion-free, 
then for $u\in\fp\cap U$ one has $(0:\fp R_\fp)_{M_\fp}=((0:\fp)_M)_\fp
\subseteq((0:u)_M)_\fp=0$, so $\depth_{R_\fp}M_\fp>0$, by \ref{ch:depth}.
  \end{proof}

Regular rings are products of integral domains, see 
\cite[14.3 and Exercise~9.11]{Ma}, so their associated primes are minimal.
Thus, the preceding lemma specializes to the following statement:

\begin{lemma}
  \label{lem:localization2}
When $R$ is a regular ring, an $R$-module $M$ is torsion-free if and 
only if $M_\fp$ is torsion-free over $R_\fp$ for each $\fp\in\Spec R$.
 \qed
  \end{lemma}

In the proof of Theorem \ref{thm:torsion}, we  use the
following result, due to Huneke and Wiegand~\cite[6.3]{HW}. An
extension to all modules $\HH i{\bsx; L}$ is given by Takahashi
et.~al~\cite[Thm.\,2]{TTKH}; it sharpens the classical rigidity theorem
recalled in \ref{koszul:rigidity}.

\begin{bfchunk}{Koszul Rigidity.~II.}
\label{ch:koszul2}
Let $(S,\fn)$ be a local ring, $\bsx$ a finite sequence of elements
of $\fn$, and $L$ a  finite $S$-module.

If $0< \length_{S}\HH 1{\bsx; L}<\infty$ holds, then
$\depth_{S}\HH 0{\bsx; L}=0$.
 \end{bfchunk}

\stepcounter{theorem}

\begin{proof}[Proof of Theorem~\emph{\ref{thm:torsion}}] 
By hypothesis, $R$ is essentially smooth over a field, say $K$. 
Let $A$ and $B$ be witnesses for $M$ and $N$, respectively, and set
\[
Q={R\otimes_KR} \quad\text{and}\quad C=A\otimes_{K}B\,.
\]
The $Q$-algebra $C$ is essentially of finite type, and $\tor iRMN$ is a finite 
$C$-module for each $i$; see Lemma~\ref{lem:eft-tors}.  Recall the hypothesis: 
$\zdr R(M\otimes_RN)=0$.

\medskip

(1) In view of Theorem~\ref{thm:rigidity}, it suffices to prove that $\tor 1RMN$ is zero.

Suppose it is not. Pick a prime ideal $\fm$ in $C$ such that $\length_{C}{\tor 1RMN}_{\fm}$ is non-zero and finite. To get a contradiction, it suffices to show that for the ideal $\fq=\fm\cap Q$ one has
   \begin{equation}
  \label{eq:sufficient}
     \tag{2.9.1}
\depth_{Q_{\fq}}(M\otimes_{R}N)_{\fq}=0\,.
  \end{equation}

Indeed, set $\fp=\fq R$.  As $Q$ acts on $(M\otimes_{R}N)_\fq$ via the map $Q_\fq\to {R_\fp}$, Lemma~\ref{lem:depth-change}(1) and formula \eqref{eq:sufficient} give $\depth_{R_{\fq}}(M\otimes_{R}N)_{\fq}=0$.
Note that $\tor 1{R_\fp}{M_{\fq}}{N_{\fq}}$ is non-zero, as it is isomorphic to 
the module $\tor 1RMN_{\fq}$, which localizes to $\tor 1RMN_\fm\ne0$.  
In particular, ${R_\fp}$ is not a field.  Since $R$ is reduced,  this implies $
\fp\notin\Ass R$, so Lemma \ref{lem:localization2} gives the desired contradiction.

It remains to establish \eqref{eq:sufficient}.  Let $\bsx$ be a minimal generating set for 
the kernel of the homomorphism $Q_{\fq}\to {R_\fp}$.  Lemma \ref{lem:smooth}
and localization give
  \begin{equation}
    \label{eq:torsion}
     \tag{2.9.2}
  \tor jRMN_{\fm} 
  \cong \HH j{\bsx; (M\otimes_KN)_{\fm}}
  \end{equation}
as ${R_\fp}$-modules.  It follows that $\HH 1{\bsx; (M\otimes_KN)_{\fm}}$ has 
non-zero finite length over $C_\fm$.  In view of \ref{ch:koszul2}, this 
gives the second equality in the string
  \[
\depth_{C_{\fm}}(M\otimes_{R}N)_{\fm}
= \depth_{C_{\fm}}\HH 0{\bsx;(M\otimes_KN)_{\fm}}=0\,.
  \]
The first one comes from \eqref{eq:torsion} with $j=0$.  Lemma 
\ref{lem:depth-change}(1), applied to the local homomorphism 
$Q_{\fq}\to C_{\fm}$, gives $\depth_{Q_{\fq}}(M\otimes_{R}N)_{\fm}=0$. 
Now \eqref{eq:sufficient} results from Lemma \ref{lem:depth-change}(2) and the 
isomorphism $(M\otimes_{R}N)_{\fm}\cong(M\otimes_{R}N)_{\fq}\otimes_{C_\fq}C_\fm$.

\medskip

(2) By symmetry, it suffices to prove $\tor iR{\zdr RM}N=0$ for each $i$. 

The case $i=0$ is settled by Lemma~\ref{lem:zdr}, in view of (1). 

Observe that $\zdr RM$ is an $A$-submodule of $M$, so the exact sequence 
\eqref{eq:zdr} is one of $A$-modules.   In particular, $\tf R M$ is a finite $A$-module, 
and hence is essentially of finite type over $R$.  Our hypothesis and Lemma~\ref{lem:zdr} 
imply that $(\tf R{M})\otimes_{R}N$ is torsion-free over $R$, so the already 
established part (1) yields 
\[
\tor iRMN = 0 = \tor iR{\tf R{M}}N\quad\text{for all $i\geq 1$.}
\]
It now follows from \eqref{eq:zdr} that $\tor iR{\zdr R{M}}N=0$ holds for $i\geq 1$.
\end{proof}

We show that when $M$ and $N$ are modules essentially of finite type over $R$,
their torsion-freeness is not related to that of $M\otimes_{R}N$, in general.
 
\begin{example}
\label{ex:tensor1}
Let $R$ be a domain and let $\fp\ne0$ and $\fq$ be prime ideals of $R$,
with $\fp\nsubseteq\fq$.  The $R$-module $M=R/\fp$ is finite and torsion, and 
the $R$-module $N=k(\fq)$ is essentially of finite type, as witnessed by $R_{\fq}$.

One has $M\otimes_{R}N=0$, so $M\otimes_{R}N$ is torsion-free.  On the other 
hand, the module $N$ is torsion when $\fq\ne0$, and is torsion-free when $\fq=0$.
  \end{example}

 \section{Torsion in tensor powers}
   \label{sec:torsion2}

In Example \ref{ex:tensor1}, we noted that a tensor \emph{product} 
of modules essentially of finite type may be torsion-free, while its 
factors need  not have this property.  Here we prove that such a 
situation does not occur for tensor \emph{powers}:

\begin{theorem}
\label{thm:tf-descent}
Let $R$ be an essentially smooth algebra over a field, $M$ an 
$R$-module essentially of finite type, and $d$ a positive integer.

If the $R$-module $\tensor dM$ is torsion-free, then so is 
$\tensor nM$ for $1\leq n\leq d$.
 \end{theorem}

  \begin{Notes}
This step is barely visible in the proof of Auslander's theorem:  When $R$ is a regular local ring
and $M$ is a finite $R$-module with $\tensor dM$ torsion-free, \cite[3.1(b)]{Au} 
and \cite[Cor.\,2(b)]{Li} give ${\zdr R{(\tensor nM)}}\otimes_R{\tensor{d-n}M}=0$ 
for $1\leq n\leq d$, whence ${\zdr R{(\tensor nM)}}=0$, by Nakayama's Lemma.
  \end{Notes}
  
The hypothesis on $M$ cannot be weakened too much; see 
Remark \ref{ex:noneft}.

The \emph{support} of a module $M$ over a commutative ring 
$A$ is it the set
\[
\Supp_{A}M = \{\fm\in\Spec A\mid M_{\fm}\ne 0\}\,.
\]
Note that one has $\Supp_{A}M=\varnothing$ if and only if $M=0$.

\begin{lemma}
\label{lem:eft-tensor0}
Let $R$ be a noetherian ring, $M$ and $N$ be $R$-modules essentially of 
finite type, and $A$ an $R$-algebra that is a witness for $M$ and $N$.

If $\Supp_{A}M\cap \Supp_{A}N\neq\varnothing$ holds, then $M\otimes_{R}N\ne 0$.
  \end{lemma}

\begin{proof}
Since $\Supp_{A}(M\otimes_{A}N) = \Supp_{A} M \cap \Supp_{A}N$,
one has $M\otimes_{A}N\ne0$, and hence $M\otimes_{R}N\ne0$, due to the 
surjection $M\otimes_{R}N\to M\otimes_{A}N$.
  \end{proof}

In the preceding lemma, the hypothesis that a \emph{common} witness 
exists is necessary, because a tensor product of modules essentially 
of finite type may be zero otherwise, even when the ring $R$ is local; 
see Example \ref{ex:tensor1}.

Recall that $\zdr RN$ denotes the $R$-torsion submodule of $N$.

\begin{proof}[Proof of Theorem~\emph{\ref{thm:tf-descent}}]
It suffices to prove that if $\zdr R{(\tensor {n+1}M)}=0$ holds, then  
$\zdr R{(\tensor {n}M)}=0$ holds as well.

Theorem~\ref{thm:torsion}(2), the isomorphism $\tensor nM\otimes_{R}M\cong\tensor {n+1}M$ 
and the hypothesis $\zdr R{(\tensor {n+1}M)}=0$ imply $\zdr R{(\tensor{n}M)}\otimes_{R} M=0$.
Thus, we obtain 
  \begin{align*}
\zdr R{(\tensor nM)} \otimes_{R} \tensor {n}M 
    &\cong \zdr R{(\tensor nM)} \otimes_{R}  \big(M\otimes_{R} \tensor {n-1}M\big)  \\
    & \cong \big(\zdr R{(\tensor{n}M)} \otimes_{R} M\big) \otimes_{R}\tensor {n-1}M\\
    & =0\,.
  \end{align*}

Choose a witness $A$ for $M$ and set $B=\tensor {n}A$. The $R$-algebra $B$ is a 
witness for $\tensor{n}M$, and hence also for $\zdr R{(\tensor nM)}$.  In view of 
Lemma~\ref{lem:eft-tensor0}, the equality above implies the second one of the following equalities 
  \[
\Supp_{B}(\zdr R{(\tensor nM)}) = \Supp_{B}(\zdr R{(\tensor nM))}\cap\Supp_{B}(\tensor nM)
=\varnothing\,.
  \]
The first one holds as $\zdr R{(\tensor nM)}$ is a submodule of $\tensor nM$.
So $\zdr R{(\tensor nM)}=0$.
  \end{proof}

\begin{remark}
\label{ex:noneft}
Let $M$ be an $R$-module.  Recall that $M$ is said to be \emph{divisible} if the 
homothety $m\mapsto um$ is surjective for every non-zero-divisor $u\in R$.  

If $M$ is non-zero, torsion and divisible, then $M\otimes_{R}M=0$. Such an $M$ is not essentially of finite type; the conclusion of Theorem \ref{thm:tf-descent} fails for it.

As an example, let $K$ be a field, set $R=K[x]$, and take $M=K(x)/K[x]$.
  \end{remark}

\section{Flatness}
  \label{sec:flatness}

Here we prove the Main Theorem, announced in the introduction:

\begin{theorem}
\label{thm:Mflat}
Let $R$ be an essentially smooth algebra over a field and let $M$ be
an $R$-module essentially of finite type.

If $\tensor dM$ is torsion-free for some integer $d\geq \dim R$, then 
$M$ is flat. 
\end{theorem}

  \begin{Notes}
When $R$ is a regular local ring and $M$ is a finite $R$-module, the 
conclusion of Theorem \ref{thm:Mflat} is established by Auslander \cite[3.2]{Au} when $R$ 
is unramified and by Lichtenbaum \cite[Cor.\,3]{Li} in general.  It 
is deduced from the analogs of Theorems \ref{thm:torsion} and
\ref{thm:tf-descent}, by using the additivity of projective dimensions
on $\operatorname{Tor}$-independent modules.

We employ a similar technique to deduce Theorem \ref{thm:Mflat} from 
Theorems \ref{thm:torsion} and~\ref{thm:tf-descent}.  However, projective 
dimensions are not always additive for non-finite modules, so we
replace them with an invariant that has the desired property.
  \end{Notes}

When $(S,\fn,l)$ is a local ring, for each $S$-module $M$ one sets
  \[
\cd SM = \sup\{i\in\BZ\mid\tor iS{l}M\ne 0\}\,.
  \]
In particular, when $\tor iSlM=0$ for all $i$, one has $\cd SM= -\infty$.

The name \emph{codepth} is motivated by the description of depth 
in terms of $\operatorname{Ext}$, and by the well-known result below, 
whose proof is included for completeness.

\begin{lemma}
\label{lem:codepth}
When $(S,\fn,l)$ is a regular local ring and $M$ an $S$-module, 
   \begin{equation}
    \label{eq:codepth2}
     \tag{4.2.1}
\cd SM = \dim S-\depth_SM\,.
  \end{equation}

If $N$ is an $S$-module with $\tor iSMN=0$ for $i\ge1$, then 
  \begin{equation}
    \label{eq:codepth3}
     \tag{4.2.2}
\cd S{(M\otimes_{S}N)}=\cd SM + \cd SN\,.
  \end{equation}
\end{lemma}

\begin{proof}
Let $\bss$ be a minimal generating set for $\fn$. As $S$ is regular, the  Koszul 
complex on $\bss$  is a free resolution of $l$, and hence $\tor iSlL\cong \HH i{\bss\,; L}$. 
A comparison of the definitions of
depth (see~\ref{ch:depth}) and codepth validates \eqref{eq:codepth2}.

Let $F$ and $G$ be flat resolutions of $M$ and $N$, respectively. The hypothesis 
translates to the statement that the complex $F\otimes_{S}G$ of flat $S$-modules 
is a resolution of $M\otimes_{S}N$. This  gives rise to the first isomorphism below:
\begin{align*}
\tor *Sl{M\otimes_{S}N}
	&\cong \HH *{l\otimes_{S}(F\otimes_{S}G)} \\
	&\cong \HH *{(l\otimes_{S}F)\otimes_{l}(l\otimes_{S}G)}\\
	&\cong \HH *{l\otimes_{S}F} \otimes_{l} \HH *{l\otimes_{S}G}\\
	&\cong\tor *SlM\otimes_{l}\tor *SlN
\end{align*}
The second one is standard; the third one is the K\"unneth isomorphism.  
Now equate the highest degree in which a vector space on either side
is non-zero.
\end{proof}

When a module $M$ over a noetherian ring $R$ has a finite flat 
resolution, and $\fd_RM$ denotes the shortest length of such a 
resolution, one has
\[  
\fd_RM=\sup\{\cd {R_{\fp}}{M_\fp}\mid\fp\in \Spec R\}
\]
by Chouinard \cite[1.2]{Ch}.  We use only the following special case of this result.

\begin{lemma}
\label{lem:flattest}
Let $R$ be a regular ring and $M$ an $R$-module. 

If $\cd{R_{\fp}}{M_{\fp}}\leq 0$ holds for each $\fp\in\Spec R$, then $M$ is flat.
\end{lemma}

\begin{proof}
Since $M$ is flat when $M_{\fp}$ is flat for each $\fp\in \Spec R$, we may 
assume that $R$ is a regular local ring.  Every $R$-module then has finite 
flat dimension, see \cite[19.2]{Ma}, so Chouinard's formula, recalled above,
gives the result.
\end{proof}

\begin{proof}[Proof of Theorem~\emph{\ref{thm:Mflat}}]
By Lemma \ref{lem:flattest}, it suffices to fix $\fp\in\Spec R$ and prove 
$\cd{R_{\fp}}{M_{\fp}}\leq 0$.  It follows from the definitions that $R_{\fp}$ is 
essentially smooth over a field and $M_{\fp}$ is essentially of finite type over 
$R_{\fp}$.  The $R_\fp$-module $\tensor[R_\fp]d{(M_\fp)}$ is isomorphic to 
$(\tensor dM)_{\fp}$, and hence it is torsion-free by Lemma \ref{lem:localization2}.  
Thus, we assume that $R$ is local with maximal ideal $\fp$, and we set out 
to prove that if $\tensor dM$ is torsion-free for some $d\geq \dim R$, then $\cd RM\leq 0$.  

If $\dim R=0$, then $R$ is a field, and the assertion is obvious.  

Assume $\dim R\ge1$.  Theorem~\ref{thm:tf-descent} shows that the 
$R$-module $\tensor nM$ is torsion-free for $1\leq n\leq d$. Thus
Theorem~\ref{thm:torsion}(1), applied with $N=\tensor {n-1}M$, yields
 \[
\tor iRM{\tensor {n-1}M}=0 \quad\text{for each $n$ with $2\leq n\leq d$ and each $i\geq 1$.}
 \]
Repeated application of formula \eqref{eq:codepth3} then gives the 
equality below:
 \[
d\cd RM = \cd R{(\tensor dM)} \leq \dim R - 1 < d\,.
 \]
The first inequality comes from \eqref{eq:codepth2} and Lemma 
\ref{lem:localization}, as $\tensor dM$ is torsion-free; the second 
one holds by hypothesis.  As a result, we get $\cd {R}M\leq 0$.
 \end{proof}

A reformulation of Theorem~\ref{thm:Mflat} gives the geometric criterion 
for flatness, stated in the introduction:

\begin{remark}
With $R$ and $M$ as in Theorem~\ref{thm:Mflat}, let $B$ be a witness for $\tensor dM$ and 
$\beta\col R\to B$ the structure homomorphism; for example, set 
$B=\tensor dA$, with $A$ a witness for $M$. For the induced map 
${}^{\mathsf{a}}\beta\col \Spec B\to\Spec R$, one has:

\smallskip

\emph{If\,\ ${}^{\mathsf{a}}\beta(\Ass_B(\tensor dM))$ is contained in 
$\Ass R$, then $M$ is flat over $R$.}

\smallskip

Indeed,  \cite[Ex.\,6.7]{Ma} gives ${}^{\mathsf{a}}\beta(\Ass_B(\tensor dM))
=\Ass_{R}(\tensor dM)$, so the hypothesis yields 
$\Ass_{R}(\tensor dM)\subseteq\Ass R$, hence $\tensor dM$ is torsion-free 
over $R$; see Lemma \ref{lem:localization}.
  \end{remark}

\section{Dimension two}
  \label{sec:two}

It is natural to ask whether the conclusion of Theorem 
\ref{thm:Mflat} holds under the weaker assumptions that $R$ is a
regular ring of finite Krull dimension and $M$ is module that is finite 
over a noetherian $R$-algebra. We give a positive answer when 
$d=\dim R\le2$, by extending the argument used by Vasconcelos 
to prove the case $M=A$; see \cite[6.1]{Va}.

\begin{proposition}
\label{prop:vas}
Let $R$ be a regular ring with $\dim R\leq 2$, let $A$ be a noetherian
$R$-algebra, and let $M$ be a finite $A$-module.

If the $R$-module $M\otimes_{R}M$ is torsion-free, then $M$ is flat
over $R$. 
 \end{proposition}

\begin{proof}
It suffices to fix $\fp$ in $\Spec R$ and prove $\cd {R_{\fp}}{M_{\fp}}\leq 0$; see Lemma~\ref{lem:flattest}. Using
Lemma~\ref{lem:localization2}, one can reduce to the case where $R$
is local, with $\fp$ its maximal ideal, so the desired result is that
$\cd RM\leq 0$. We may  assume $\dim R\geq 1$.

It is enough to prove that $M$ is torsion-free and that there are equalities:
\[
\tor iRMM=0\quad\text{for}\quad i=1,2\,.
\]
Indeed, since $R$ is regular with $\dim R\leq 2$, they imply $\tor iRMM=0$ for each $i\geq 1$. One has $\cd R{(M\otimes_{R}M)} \leq 1$, by \eqref{eq:codepth2} and Lemma \ref{lem:localization}, so Lemma~\ref{lem:codepth}(2), applied with $N=M$, yields $\cd RM\leq 0$, as desired.

Recall that $\zdr RM$ is the $R$-torsion submodule of $M$ and $\tf RM=M/\zdr RM$.  As $\tf RM$ is torsion-free and
$R$ is a domain, there is an exact sequence
 \[
0\lra \tf RM\lra U^{-1}(\tf RM)\lra C\lra 0
 \]
of $R$-modules, where $U = R\smallsetminus\{0\}$.  As $U^{-1}(\tf RM)$ is flat over $R$,
and one has $\fd_RC\le2$, we obtain $\fd_R(\tf RM)\le1$.  This gives 
the equalities below:
\[
\tor 1R{\tf RM}{\tf RM}\cong\tor 2R{\tf RM}C=0=\tor 2R{\tf RM}{\tf RM} \,.
 \]
The isomorphism is obtained by tensoring the sequence above with $\tf RM$.
To finish, we prove that $M$ is torsion-free; that is to say, $\zdr RM=0$ holds. 

By way of contradiction, assume $\zdr RM\ne0$.  As $\zdr RM$
is an $A$-module, we have $(\zdr RM)_{\fm}\ne 0$ for some $\fm$ 
in $\Spec A$. There is a natural isomorphism of $A_{\fm}$-modules 
$(\zdr RM)_{\fm}\cong \zdr R(M_{\fm})$, and 
$M_{\fm}\otimes_{R}M_{\fm}$ is torsion-free over $R$, as it is a localization 
of $M\otimes_{R}M$.  Thus, we may also assume that $A$ is local.

Lemma~\ref{lem:zdr}, applied with $N=M$, shows that $M\otimes_{R}\tf RM$ 
is torsion-free. As $\tor 1R{\tf RM}{\tf RM}=0$ holds, the last assertion in 
Lemma~\ref{lem:zdr}, now applied with $N=\tf RM$, gives 
$\zdr R{M}\otimes_{R}\tf R{M}=0$.  Note that $\zdr RM$, being an submodule 
of the finite $A$-module $M$, is itself finite, so $A$ is a witness for both 
$\zdr RM$ and $\tf R{M}$.  As the maximal ideal of $A$ is contained in
the support of every non-zero finite $A$-module, and $\zdr RM$ is non-zero, 
Lemma \ref{lem:eft-tensor0} implies $\tf RM=0$.  Thus, $M$ is a torsion 
$R$-module, and then so is $M\otimes_{R}M$. 

This contradicts our hypothesis.
 \end{proof}

\section*{Acknowledgment}
This paper was triggered by discussions with Janusz Adamus.  We thank him, 
Ed Bierstone, and Pierre Milman for correspondence concerning \cite{ABM}.

\end{document}